\documentclass[12pt]{amsart}
\usepackage{a4}
\usepackage{amssymb}
\usepackage[dvipdfm]{color}

\theoremstyle{plain}
\newtheorem{Def}{Definition}[section]
\newtheorem{Prop}[Def]{Proposition}
\newtheorem{Thm}[Def]{Theorem}
\newtheorem{Lem}[Def]{Lemma}
\newtheorem{Cor}[Def]{Corollary}

\newtheorem{Remark}[Def]{Remark}

\numberwithin{equation}{section}

\def\fr#1{\mathfrak{#1}}

\newcommand{\rank}{\mathop{\mathrm{rank}}\nolimits}
\newcommand{\ad}{\mathop{\mathrm{ad}}\nolimits}
\newcommand{\Lie}{\mathop{\mathrm{Lie}}\nolimits}

\newcommand{\Isom}{\mathop{\mathrm{Isom}}\nolimits}

\title[A sufficient condition for congruency of orbits]{
	A sufficient condition for congruency of orbits of Lie groups and some applications}
\author[A.\ Kubo]{Akira Kubo}
\address[A.\ Kubo]{Department of Mathematics, Hiroshima University, Higashi-Hiroshima 739-8526, Japan}
\email{akira-kubo@hiroshima-u.ac.jp}
\author[H.\ Tamaru]{Hiroshi Tamaru}
\address[H.\ Tamaru]{Department of Mathematics, Hiroshima University, Higashi-Hiroshima 739-8526, Japan}
\email{tamaru@math.sci.hiroshima-u.ac.jp}

\thanks{The second author was supported in part by KAKENHI (20740040, 24654012).}

\begin{document}

\begin{abstract}
We give a sufficient condition for isometric actions 
to have the congruency of orbits, 
that is, all orbits are isometrically congruent to each other. 
As applications, we give simple and unified proofs for some known congruence results, 
and also provide new examples of isometric actions on symmetric spaces of noncompact type 
which have the congruency of orbits.
\end{abstract}

\maketitle

\section{Introduction} 

Isometric actions of Lie groups on Riemannian manifolds $M$ 
and submanifold geometry of their orbits are fundamental topics in geometry. 
In this paper, we consider isometric actions which have the congruency of orbits, 
that is, all of whose orbits are isometrically congruent to each other. 
The congruency of orbits yields good benefits 
for studying submanifold geometry of orbits, 
since it is sufficient to study only one orbit.
It has been known that the following isometric actions have the congruency of orbits:
\begin{enumerate}
 \item the actions of $\mathrm{U}(1)$ on spheres $\mathbb{S}^{2n+1}$ which induce the Hopf fibrations, 
 \item the actions of $N$ on hyperbolic spaces which induce the horosphere foliations, 
 \item the actions of $S_V$ on symmetric spaces of noncompact type $M = G/K$,
       where $S_V$ are some codimension one subgroups of $AN$ (\cite[Proposition 3.1]{BT03}, see Section~4 for details), and 
 \item the actions of $N$ on symmetric spaces of noncompact type $M=G/K$ which induce horocycle foliations (\cite[Corollary 6.5]{BDT}). 
\end{enumerate}
Note that, for a symmetric space of noncompact type $M$, 
we denote by $G$ the identity component of the isometry group of $M$,
and by $G=KAN$ the Iwasawa decomposition. 

In this paper, we obtain a sufficient condition for isometric actions 
to have the congruency of orbits (Lemma \ref{main}). 
Indeed our sufficient condition
is stated in terms of Lie algebras, and it is very practical to apply. 

The first applications of our sufficient condition are
simple and unified proofs for the congruency of orbits 
for all of the above mentioned actions. 
In Section~3 we show the congruency of orbits in the case of the Hopf fibrations (1).
In Section~4 we 
prove the congruency of orbits for a class of actions, 
which contains the actions (2), (3) and (4). 

As the second application of our sufficient condition, in 
Section~5, we 
provide new examples of isometric actions which have the congruency of orbits. 
Namely, we
prove the congruency of 
orbits
of $S_\Phi$ on symmetric spaces of noncompact type $M = G/K$ 
(Proposition \ref{orbits_S_Phi}), 
where 
$S_\Phi$ denotes the solvable part of 
the
parabolic subgroup 
$Q_\Phi$ of $G$. 
Recently, $S_\Phi$
have played very important roles in studying submanifolds in $M$ 
(\cite{BDT, BT10, T}, 
see also a survey \cite{Tsurvey}). 
Among others, it is remarkable
that the orbit $(S_\Phi).o$ is always minimal in $M$ and 
is Einstein with respect to the induced metric (\cite{T}), 
where $o$ is 
the origin of $M = G/K$. 
Hence, 
as a corollary of the congruency of orbits of $S_{\Phi}$, 
we have the following: 
any orbits of $S_\Phi$ are minimal submanifolds in $M$,
and are Einstein with respect to the induced metrics. 

We are interested in studying further geometric properties of 
$S_{\Phi}$-orbits. 
For the study, the congruency of $S_{\Phi}$-orbits is quite useful, 
since we have only to consider one orbit. 
Furthermore, our sufficient condition would be 
useful
for further studies on isometric actions on symmetric spaces of noncompact type, 
since some interesting actions do satisfy our sufficient condition, 
and hence applicable to study of geometry of their orbits. 

Throughout this paper, 
we denote by $\Isom(M)$ the isometry group of a Riemannian manifold $M$, 
and by $\Lie(G)$ the Lie algebra of a Lie group $G$.

\section{A key lemma} 

In this section, we give a sufficient condition for isometric actions to have the congruency of orbits on Riemannian manifolds.

\begin{Lem} \label{main}
Let $M$ be a Riemannian manifold and
 $S$ be a connected Lie subgroup of $\Isom(M)$ with $\Lie(S) = \fr{s}$, 
 and assume 
 that $S$ acts transitively on $M$.
If $\fr{s}'$ is an ideal of $\fr{s}$, 
 then all orbits of $S'$ in $M$ are isometrically congruent to each other, 
 where $S'$ is the connected Lie subgroup of $S$ with $\Lie(S') = \fr{s}'$.
\end{Lem}

\begin{proof}
Take any $p, q \in M$.
We shall show that the orbits $S' . p $ and $S' . q $ are isometrically congruent.
Owing to transitivity of the action of $S$, there exists $g \in S$ such that $p = g.q$ holds.
Since
 $S$ and $S'$ are connected and $\fr{s'}$ is an ideal in $\fr{s}$, 
 one knows that
 $S'$ is a normal subgroup of $S$ (see for instance \cite[Theorem 3.48]{W}). 
Hence one has $g^{-1} S' g = S'$. 
Thus we obtain 
	\begin{align*}
        g. (S'.q) = g. (g^{-1} S' g). q = S'. (g. q) = S'.p, 
        \end{align*}
 which implies $S' . p $ and $S' . q $ are isometrically congruent. 
\end{proof}

\section{Hopf fibrations} 

In this section, 
by applying
Lemma~\ref{main}, 
 we give a simple proof for
the congruency of orbits of the actions 
$\mathrm{U}(1)$ on spheres $\mathbb{S}^{2n+1}$
which induce Hopf fibrations.

Let $\mathbb{S}^{2n+1}$ be the unit sphere in $\mathbb{C}^{n+1}$. 
Consider
the natural action of the unitary group $\mathrm{U}(n+1)$ on $\mathbb{S}^{2n+1}$, 
and let $\mathrm{U}(1)$ be the center of $\mathrm{U}(n+1)$. 
It is well-known that the orbit space of the action of $\mathrm{U}(1)$ 
on $\mathbb{S}^{2n+1}$ satisfies
\begin{align*}
\mathrm{U}(1) \backslash \mathbb{S}^{2n+1} 
= \mathrm{U}(1) \backslash \mathrm{U}(n+1) / \mathrm{U}(n) 
= \mathrm{U}(n+1) / ( \mathrm{U}(1) \times \mathrm{U}(n)) 
= \mathbb{C}\textrm{P}^n . 
\end{align*}
The natural projection 
from $\mathbb{S}^{2n+1}$ onto $\mathbb{C}\textrm{P}^n$ provides
the \textit{Hopf fibration}. 

We now show the congruency of orbits of the action above, as an application of Lemma~\ref{main}.

\begin{Prop} \label{Hopf}
Under the action of $\mathrm{U}(1)$ on $\mathbb{S}^{2n+1}$ defined above, all orbits of $\mathrm{U}(1)$ in $\mathbb{S}^{2n+1}$ are isometrically congruent to each other.
\end{Prop}

\begin{proof}
Recall that $\mathrm{U}(n+1)$ acts transitively on $\mathbb{S}^{2n+1}$.
We know that $\mathfrak{u}(1)$ is an ideal in $\mathfrak{u}(n+1)$, or that $\mathrm{U}(1)$ is a normal subgroup of $\mathrm{U}(n+1)$.
Hence 
the proof easily follows from 
Lemma~\ref{main}. 
\end{proof}

When $n=1$, the proof of Proposition~\ref{Hopf} can be found, 
for example, in \cite[Section 2]{S}. 
We note that its proof depends on the fact that 
$\mathbb{S}^{3}$ can be identified with $\mathrm{Sp}(1)$ 
equipped with the bi-invariant metric. 
Hence, our sufficient condition gives another proof, 
which can also be applied to an arbitrary $n$. 

\section{Horospheres and their generalizations} 

In this section, we give 
further applications
of Lemma~\ref{main}, 
which 
provide
the congruency of orbits of certain isometric actions 
on Riemannian symmetric spaces of noncompact type and of arbitrary rank. 
The result of this section
contains, as special cases,
simple and unified proofs of
the congruency of orbits of (2), (3) and (4) mentioned in Section~1.

First of all, we recall some fundamental notions of symmetric spaces of noncompact type. 
Refer to \cite{H, K}. 
Let $M = G/K$ be a connected Riemannian symmetric space of noncompact type, 
 where $G$ is the identity component of $\Isom(M)$, 
 and $K$ is the isotropy subgroup of $G$ at some point $o$, called the origin.
Let us denote by $\fr{g}$ and $\fr{k}$ the Lie algebras of $G$ and $K$, respectively,
 and by $\fr{p}$ the orthogonal complement of $\fr{k}$ with respect to the Killing form $B$ of $\fr{g}$.
One thus obtains that $\fr{g} = \fr{k} \oplus \fr{p}$, the Cartan decomposition of $\fr{g}$.
Denote by $\theta$ the corresponding Cartan involution.
We then introduce a positive definite inner product on $\fr{g}$ by $\langle X, Y \rangle := -B(X, \theta Y)$. 

Let $\fr{a}$ be a maximal abelian subspace of $\fr{p}$ and denote the dual space of $\fr{a}$ by $\fr{a}^*$.
Then we define 
	\begin{align*}
	\fr{g}_\lambda := \{ X \in \fr{g} : 
	\ad(H)X = \lambda(H)X \text{ for all } H \in \fr{a} \}
	\end{align*} 
 for each $\lambda \in \fr{a}^*$, 
 and call $\lambda \in \fr{a}^* \setminus \{ 0 \}$ the \textit{restricted root} if $\fr{g}_\lambda \ne {0}$.
Denote by $\Sigma$ the set of restricted roots.
Let $\Lambda$ be a set of simple roots of $\Sigma$,
 and then denote by $\Sigma^+$ the set of positive roots associated with $\Lambda$.

Let us define 
	\begin{align*} \textstyle 
	\fr{n} := \bigoplus_{\lambda \in \Sigma^+} \fr{g}_\lambda , \quad 
	\fr{s} := \fr{a} \oplus \fr{n} , 
	\end{align*}
which yields that $\fr{g} = \fr{k} \oplus \fr{a} \oplus \fr{n}$, 
the Iwasawa decomposition of $\fr{g}$. 
Note that
$\fr{n}$ is a nilpotent subalgebra and $\fr{s}$ is a solvable subalgebra. 
Furthermore, one can check easily that
	\begin{align} \label{eq_2_1}
	[\fr{s}, \fr{s}] = \fr{n}.
	\end{align}
Let $S$ be the connected Lie subgroup of $G$ with $\Lie(S) = \fr{s}$.
This solvable subgroup $S$ is simply-connected and acts simply transitively on $M$.

We now give examples of isometric actions on $M$, which have the congruency of orbits. 
Denote by $\ominus$ the orthogonal complement with respect to $\langle , \rangle$.

\begin{Prop} \label{orbits_S_V}
Let $V$ be any linear subspace of $\fr{a}$ and 
 define $\fr{s}_V := \fr{s} \ominus V = (\fr{a} \ominus V) \oplus \fr{n}$.
Then all orbits of $S_V$ in $M$ are isometrically congruent to each other, 
 where $S_V$ is the connected Lie subgroup of $S$ with $\Lie(S_V) = \fr{s}_V$.
\end{Prop}

\begin{proof}
Recall that $S$ acts transitively on $M$. 
Then, by Lemma~\ref{main}, we have only to prove that $\fr{s}_V$ is an ideal in $\fr{s}$.
It follows easily from (\ref{eq_2_1}) and the definitions of $\fr{s}$ and $\fr{s}_V$ that 
	\begin{align*}
    	[\fr{s}, \fr{s}_V] \subset [\fr{s}, \fr{s}] = \fr{n} \subset \fr{s}_V . 
        \end{align*}
Thus $\fr{s}_V$ is an ideal in $\fr{s}$. 
This completes the proof.
\end{proof}

The actions of $S_V$ on $M$ have been studied in \cite{BDT}, 
 and are always hyperpolar. 
Furthermore, the class of $S_V$-actions contains some interesting subclasses. 

\begin{Remark}
Although the proof of Proposition~\ref{orbits_S_V} is very easy, 
it gives simple and unified proofs of some known results as follows. 
\begin{enumerate}
\item
In the case when $\rank M = 1$ and $\dim V = 1$, 
 Proposition~\ref{orbits_S_V} gives a proof of the congruency of horospheres
in hyperbolic spaces. 
Indeed horospheres coincide with $N$-orbits, 
 where $N$ stands for the connected Lie subgroup of $G$ with $\Lie(N) = \fr{n}$.
Hence, setting $V := \fr{a}$ we have $N = S_V$.

\item
In the case when $\rank M > 1$ and $\dim V = 1$, 
Proposition~\ref{orbits_S_V} gives another proof of \cite[Proposition 3.1]{BT03}. 
We note that the proof in \cite{BT03} involves some geometric arguments, 
 but our proof is purely Lie algebraic.
We also note that the action of $S_V$ on $M$ is of cohomogeneity one.

\item
In the case when $\rank M > 1$ and $\dim V = \rank M$, 
 Proposition~\ref{orbits_S_V} also gives another proof of \cite[Corollary 6.5]{BDT}. 
Note that $S_V = N$ in this case, whose orbits form the horocycle foliation. 
\end{enumerate}
\end{Remark}

In the case when $\rank M > 1$ and $\dim V$ generic,
 Proposition~\ref{orbits_S_V} is a slight extension of the above mentioned results.

\section{The solvable parts of parabolic subgroups} 

In this section we introduce 
 new
examples of isometric actions having the congruency of orbits on Riemannian symmetric spaces of noncompact type.
They are induced by the solvable parts of parabolic subalgebras. 

We first review parabolic subalgebras (we refer to \cite{K}).
We use the notations in Section~4. 
It is known that there
is a one-to-one correspondence between 
proper
subsets $\Phi \subsetneq \Lambda$ 
 and the conjugacy classes of parabolic subalgebras of $\fr{g}$. 
The correspondence is given as follows. 
For each $\Phi \subsetneq \Lambda$,
 let $\Sigma_\Phi$ be the root subsystem of $\Sigma$ generated by $\Phi$, 
 that is, $\Sigma_\Phi$ is the intersection of $\Sigma$ and the linear span of $\Phi$,
 and put $\Sigma_\Phi^+ := \Sigma_\Phi \cap \Sigma^+$. 
Then, let us define
	\begin{align*}
	\textstyle 
	\fr{q}_\Phi := \fr{g}_0 \oplus 
	\left( \bigoplus_{\beta \in \Sigma_\Phi \cup \Sigma^+} \fr{g}_\beta \right) , 
	\end{align*}
 which is a parabolic subalgebra of $\fr{g}$.
It then is clear but remarkable that 
	\begin{align} \label{eq_3_1}
		\fr{s} \subset \fr{q}_\Phi.
	\end{align}
We shall give the solvable part of $\fr{q}_\Phi$. 
Consider the \textit{Langlands decomposition} $\fr{q}_\Phi = \fr{m}_\Phi \oplus \fr{a}_\Phi \oplus \fr{n}_\Phi$, 
where 
	\begin{align*}
	\fr{a}_\Phi 
	& = \{ H \in \fr{a} : \alpha(H) = 0 \text{ for all } \alpha \in \Phi \} , \\ 
	\fr{m}_\Phi 
	& = (\fr{g}_0 \ominus \fr{a}_\Phi) \oplus 
	\textstyle 
	( \bigoplus_{\beta \in \Sigma_\Phi^+} \fr{g}_\beta) , \\
	\fr{n}_\Phi 
	& = \textstyle 
	\bigoplus_{\lambda \in \Sigma^+ \setminus \Sigma_\Phi^+} \fr{g}_\lambda . 
	\end{align*}
Let us define $\fr{s}_\Phi := \fr{a}_\Phi \oplus \fr{n}_\Phi$, 
which is called the solvable part of the parabolic subalgebra $\fr{q}_\Phi$. 
By definition, one has
\begin{align*}
\fr{s}_\Phi \subset \fr{s} . 
\end{align*}
Furthermore, it follows 
from \cite[Proposition 7.78]{K} that $\fr{s}_\Phi$ is 
an ideal
in $\fr{q}_\Phi$, that is, 
	\begin{align} \label{eq_3_2}
		[\fr{s}_\Phi , \fr{q}_\Phi] \subset \fr{s}_\Phi.
	\end{align}
We are now in the position to state the main result of this section.

\begin{Thm} \label{orbits_S_Phi}
Let $\Phi \subsetneq \Lambda$, 
 and $\fr{s}_\Phi$ be the solvable part of 
 the
parabolic subalgebra $\fr{q}_\Phi$.
Then all orbits of $S_\Phi$ in $M$ are isometrically congruent to each other, 
 where $S_\Phi$ is the connected Lie subgroup of $S$ with $\Lie(S_\Phi) = \fr{s}_\Phi$.
\end{Thm}

\begin{proof}
As in Proposition~\ref{orbits_S_V}, 
we have only to check that $\fr{s}_\Phi$ is an ideal in $\fr{s}$. 
Indeed, 
it follows from (\ref{eq_3_1}) and (\ref{eq_3_2}) that 
	\begin{align}
        	[\fr{s} , \fr{s}_\Phi] \subset [\fr{q}_\Phi , \fr{s}_\Phi] \subset \fr{s}_\Phi,
        \end{align}
which completes the proof. 
\end{proof}

The proof of Theorem~\ref{orbits_S_Phi} is 
 easy, 
but it provides a lot of new examples of isometric actions on $M$
which have the congruency of orbits. 
Note that there are $2^r - 1$ proper subsets in $\Lambda$, 
where $r$ denotes 
the rank of $M$. 

We also note that the orbits $(S_\Phi).o$ through the origin $o$ have been studied in \cite{T}. 
Indeed, it is proved that $(S_\Phi).o$ is always minimal in $M$ and is Einstein with respect to the induced metric. 
Hence, combining these facts with Theorem ~\ref{orbits_S_Phi}, we readily obtain the following result. 

\begin{Cor}\label{Cor}
All orbits of $S_\Phi$ are minimal in $M$, 
and Einstein with respect to the induced metrics. 
\end{Cor}

\end{document}